\documentclass[11pt]{amsart}
\usepackage{amssymb}
\usepackage[mathscr]{eucal}
\usepackage{hyperref}
\usepackage{url}
\usepackage[all]{xy}

\makeatletter
\let\@wraptoccontribs\wraptoccontribs
\makeatother

\newtheorem{thm}[equation]{Theorem}
\newtheorem{lem}[equation]{Lemma}
\newtheorem{cor}[equation]{Corollary}
\newtheorem{prop}[equation]{Proposition}

\theoremstyle{definition}
\newtheorem{rem}[equation]{Remark}
\newtheorem{defn}[equation]{Definition}

\numberwithin{equation}{section}

\def\Z{\mathbb{Z}}
\def\Q{\mathbb{Q}}
\def\F{\mathbb{F}}

\def\bN{\mathbf{N}}

\def\Fl{\F_\ell}

\def\O{\mathcal{O}}
\def\P{\mathcal{P}}

\def\H{\mathcal{H}}

\def\K{\mathcal{K}}

\def\L{\mathcal{L}}

\def\cS{\mathcal{S}}

\def\cC{\mathcal{C}}

\def\a{\mathfrak{a}}

\def\p{\mathfrak{p}}
\def\b{\mathfrak{b}}
\def\q{\mathfrak{q}}

\def\Hom{\mathrm{Hom}}
\def\Gal{\mathrm{Gal}}
\def\rk{\mathrm{rank}}

\def\ord{\mathrm{ord}}

\def\Sel{\mathrm{Sel}}

\def\Aut{\mathrm{Aut}}
\def\loc{\mathrm{loc}}

\def\ur{\mathrm{ur}}

\def\GL{\mathrm{GL}}
\def\SL{\mathrm{SL}}
\def\ab{\mathrm{ab}}

\def\ram{\mathrm{ram}}

\def\map#1{\;\xrightarrow{#1}\;}
\def\bmu{\boldsymbol{\mu}}
\def\too{\longrightarrow}
\def\onto{\twoheadrightarrow}

\def\dirsum#1{\underset{#1}{\textstyle\bigoplus}}

\def\K{K}
\def\Hu{H^1_\ur}
\def\OK{\O_K}
\def\TT{T}

\title{Big fields that are not large}

\author{Barry Mazur}
\address{Department of Mathematics, 
Harvard University,
Cambridge, MA 02138, 
USA}
\email{\href{mailto:mazur@g.harvard.edu}{mazur@g.harvard.edu}}
\author{Karl Rubin}
\address{Department of Mathematics, 
UC Irvine,
Irvine, CA 92697, 
USA}
\email{\href{mailto:krubin@uci.edu}{krubin@uci.edu}}


\begin{document}

\begin{abstract}
A subfield $K$ of $\bar\Q$ is {\em large} if every smooth curve $C$ over $K$ with 
a $K$-rational point has infinitely many $K$-rational points.  A subfield $K$ of $\bar\Q$ 
is {\em big} if for every positive integer $n$, $K$ contains a number field $F$ with $[F:\Q]$ 
divisible by $n$.  The question of whether all big fields are large seems to have 
circulated for some time, although we have been unable to find its origin.  
In this paper we show that there are big fields that are not large.
\end{abstract}

\maketitle

\section{Introduction}

\begin{defn}
Following Pop \cite{pop}, we say that a field $K$ is {\em large} if for every smooth curve 
$C$ defined over $K$, if the set of $K$-rational points $C(K)$ is non-empty then $C(K)$ is infinite.
Equivalently (see \cite[Proposition 1.1]{pop}), $K$ is large if for every irreducible 
variety $V$ defined over $K$ with a smooth 
$K$-rational point, $V(K)$ is Zariski dense in $V$.
\end{defn}

Large fields (sometimes called {\em ample} fields) play a role in a number 
of basic conjectures regarding fields of algebraic numbers of infinite degree.   
For example, a conjecture of  Shafarevich asserts that the absolute Galois group 
of $\Q^{\ab}$, the maximal abelian extension of $\Q$, is isomorphic to
the free profinite group with countably infinitely many generators.   
It follows from a theorem of Pop \cite[Theorem 2.1]{pop} that 
if  $\Q^{\ab}$ is large, then Shafarevich's conjecture holds.
For more about large fields, see  Section \ref{add} below.

Recall that a {\em supernatural number} is a formal product 
$$
\prod_p p^{n_p}
$$
where $p$ runs through all rational primes and $0 \le n_p \le \infty$, 
with the obvious notion of divisibility.  If $K/L$ is an algebraic extension 
of fields, then we define the degree of $K$ over $L$
$$
[K:L] := \mathrm{lcm} \{[F:L] : \text{$F$ is a finite extension of $L$ in $K$}\}.
$$
\begin{defn}
We say that a field $K \subset \bar{\Q}$ is {\em big} if $[K:\Q] = \prod_p p^\infty$.
Equivalently, $K$ is big if for every positive integer $n$, $K$ contains a number 
field $F$ with $[F:\Q]$ divisible by $n$.
\end{defn}

The main result of this note is the following.

\begin{thm}
\label{main}
There are (uncountably many)  big fields that are not large.
\end{thm}

More precisely, we will exhibit a non-empty set $\cS_0$ of elliptic curves over $\Q$ 
such that for every $E \in \cS_0$, there are uncountably many big fields $K$ such 
that $E(K)$ has exactly one rational point.

\noindent
\subsection*{Acknowledgments}
Theorem \ref{main} answers a question that was brought to our attention by Arno Fehm at 
the American Institute of Mathematics workshop ``Definability and decidability problems 
in number theory'', in May 2019.  
We are grateful to him for that and for additional helpful correspondence.
We also thank the referee for very helpful comments.

\section{$\L$-towers}

\begin{defn}
\label{Ltowerdef}
Suppose $K$ is a number field, and $\L := (\ell_1,\ell_2,\ldots)$ is a  
sequence of rational primes.  The sequence $\L$ can be either infinite or finite, 
and the primes $\ell_i$ need not be distinct.  
We call an extension $K_\infty/K$ an {\em $\L$-tower} 
if there is a sequence of number fields $K = K_0 \subset K_1 \subset K_2 \subset \cdots$ 
such that
\begin{enumerate}
\item
$K_\infty = \cup_i K_i$,
\item
$K_i/K_{i-1}$ is cyclic of degree $\ell_i$,
\item
for every $i>1$ there are primes $\p_i, \p'_i$ of degree $1$ of $K_i$, lying above the 
same prime of $K_{i-1}$, such that $\p_i$ ramifies in $K_{i+1}/K_i$, but $\p'_i$ does not. 
\end{enumerate}
\end{defn}

Note that if $K_\infty/K$ is an  $\L$-tower, then $[K_\infty:K] = \prod_i \ell_i$.

\begin{lem}
\label{chunk}
Suppose the chain of number fields $K_{i-1} \subset K_i \subset K_{i+1}$ is part of an 
$\L$-tower $K_\infty/K$, in the notation of Definition \ref{Ltowerdef}.  Then there is 
no Galois extension $F/K_{i-1}$ such that $FK_i = K_{i+1}$.
\end{lem}

\begin{proof}
Suppose $F/K_{i-1}$ is Galois.  By Definition \ref{Ltowerdef} 
there are primes $\p, \p'$ of $K_i$, lying above the 
same prime $\q$ of $K_{i-1}$, such that $\p$ ramifies in $K_{i+1}/K_i$, but $\p'$ does not. 
Note that since $K_i/K_{i-1}$ is cyclic of prime degree and $\p \ne \p'$, $\q$ must split completely in 
$K_i/K_{i-1}$.  Thus the diagram
$$
\xymatrix@R=10pt@C=0pt{
&FK_i \\
F\ar@{-}[ur] && K_i  \ar@{-}[ul] \\
& K_{i-1} \ar@{-}[ur] \ar@{-}[ul]
}
$$
shows that
\begin{multline*}
\text{$\p$ ramifies in $FK_i/K_i \iff \q$ ramifies in $F/K_{i-1}$} \\
   \text{$\iff \p'$ ramifies in $FK_i/K_i$.}
\end{multline*}
Thus $FK_i \ne K_{i+1}$, and the lemma follows.
\end{proof}

\begin{lem}
\label{towerlem}
Suppose $\L$ is a sequence of rational primes, $K$ is a number field, and 
$K_\infty = \cup K_i$ is an $\L$-tower over $K$ as in Definition \ref{Ltowerdef}.  
Then:
\begin{enumerate}
\item
for every $i \ge 0$, the maximal abelian extension of $K_i$ in $K_\infty$ is $K_{i+1}$,
\item
if $T$ is a Galois extension of $K$ (possibly of infinite degree) and $K_1 \cap T = K$, 
then $K_i \cap T = K$ for every $i \ge 1$.
\end{enumerate}
\end{lem}

\begin{proof}
Suppose that there is an abelian extension $L$ of $K_i$ contained in $K_\infty$ and 
properly containing $K_{i+1}$.  Without loss of generality we may assume that 
$[L:K_{i+1}]$ is  prime.  Choose $\alpha \in K_\infty$ such that $L = K_{i+1}(\alpha)$, 
and let $n$ be such that $\alpha \in K_{n+1}$ but $\alpha \notin K_n$.  
Since $L \supsetneq K_{i+1}$ we have $n > i$.

Since $\alpha \notin K_n$ and $[K_{n+1}:K_n]$ is prime, we have $K_{n+1} = K_n(\alpha) = LK_n$.
This contradicts Lemma \ref{chunk} applied with $i = n$ and $F = LK_{n-1}$.
This contradiction proves (i).

Now suppose $T/K$ is Galois with group $G$, and $K_1 \cap T = K$.  We will show by induction that 
$K_i \cap T = K$ for every $i$.  Suppose $K_i \cap T = K$, and consider the diagram
$$
\xymatrix{
K_{i+1} && K_iT \\
K_i \ar@{-}^{G}[urr]\ar@{-}^{\ell_{i+1}}[u] && T \ar@{-}[u] \\
K \ar@{-}^{G}[urr]\ar@{-}[u]
}
$$
We either have $K_{i+1} \subset K_iT$ or $K_{i+1} \cap K_iT = K_i$.  
Suppose first that $K_{i+1} \subset K_iT$.  Let $H := \Gal(K_i T/K_{i+1}) \subset G$, 
so $T^H/K$ is abelian of degree $\ell_{i+1}$. 
Since $K_i \cap T = K$, we have $K_{i+1} = T^H K_i$.  
This contradicts Lemma \ref{chunk} applied with $F = T^H K_{i-1}$, 
so we conclude that $K_{i+1} \cap K_iT = K_i$.  Thus
$$
K_{i+1} \cap T = K_{i+1} \cap (K_iT \cap T) = (K_{i+1} \cap K_iT) \cap T = K_i \cap T = K.
$$
By induction this completes the proof of (ii).
\end{proof}

\begin{cor}
\label{cor}
Suppose $K$ is a number field and $K_\infty/K$ is an algebraic extension.  
There is at most one sequence of rational primes $\L$ and one sequence 
of fields $(K_0,K_1,\ldots)$ that exhibits $K_\infty/K$ as an $\L$-tower.
\end{cor}

\begin{proof}
Suppose $K_\infty/K$ is an $\L$-tower for some sequence $\L$.  Applying Lemma \ref{towerlem}(i) 
inductively shows that the sequence of fields $(K_0,K_1,\ldots)$ is uniquely determined, 
and then the fact that $\ell_i = [K_{i}:K_{i-1}]$ determines $\L$.
\end{proof}

\section{Selmer groups}
If $F$ is a perfect field, 
$G_F$ will denote its absolute Galois group $\Gal(\bar{F}/F)$.

For this section fix a number field $K$, an elliptic curve $E$ defined over $K$, and 
a rational prime $\ell$ such that 
\begin{equation}
\label{ess}
\text{the natural map $G_K \onto \Aut(E[\ell]) \cong \GL_2(\Fl)$ is surjective.}
\end{equation}
Let $L$ be either $K$ itself or a cyclic extension of $K$ of 
degree $\ell$.  

For every place $v$ of $K$, let $K_v$ denote the completion of $K$ at $v$, 
let $L_v$ denote the completion of $L$ at some place above $v$, and 
let $\H_\ell(L_v/K_v)$ be the ``local condition'' subspace of $H^1(K_v,E[\ell])$ defined in 
\cite[Definition 7.1]{ds}.  We will not repeat the definition here, but the following 
four propositions list all the properties of the subspaces $\H_\ell(L_v/K_v)$ that we need. 

If $v$ is nonarchimedean with residue characteristic different from $\ell$, and
$E$ has good reduction at $v$, let $K_v^\ur$ denote the maximal unramified extension of 
$K_v$ and 
$$
\Hu(K_v,E[\ell]) := H^1(K_v^\ur/K_v,E[\ell]) \subset H^1(K_v,E[\ell]).
$$

\begin{prop}
\label{hl}
\begin{enumerate}
\item
$\dim_{\Fl} \H_\ell(L_v/K_v) = \frac{1}{2} \dim_{\Fl}H^1(K_v,E[\ell])$.
\item
If $L_v = K_v$ then $\H_\ell(L_v/K_v)$ is the image of the Kummer map 
$
E(K_v)/\ell E(K_v) \to H^1(K_v,E[\ell]).
$
\end{enumerate}
\end{prop}

\begin{proof}
Assertion (i) holds because $\H_\ell(L_v/K_v)$ is a maximal isotropic subspace of 
$H^1(K_v,E[\ell])$ for the local Tate pairing (see for example \cite[Proposition 4.4]{alc}).
Assertion (ii) is explained in \cite[Definition 7.1]{ds}.
\end{proof}

\begin{prop}
\label{hl1}
If $v \nmid \ell\infty$, $E$ has good reduction at $v$, and $L_v/K_v$ is ramified, then
$\H_\ell(L_v/K_v) \cap \Hu(K_v,E[\ell]) = 0$.
\end{prop}

\begin{proof}
This is \cite[Proposition 7.8(ii)]{ds}.
\end{proof}

\begin{prop}
\label{hl2}
If $v \nmid \ell\infty$, $L_v/K_v$ is unramified, $E$ has good reduction at $v$, 
and $\phi_v \in \Gal(K_v^\ur/K)$ is the Frobenius generator, then:
\begin{enumerate}
\item
$\H_\ell(L_v/K_v) = \Hu(K_v,E[\ell])$,
\item
$\dim_{\Fl} \H_\ell(L_v/K_v) = \dim_{\Fl} (E[\ell]^{\phi_v=1})$,
\item
the map $\Hu(K_v,E[\ell]) \to E[\ell]/(\phi_v-1)E[\ell]$ given by evaluating 
cocycles at $\phi_v$ is a well-defined isomorphism.
\end{enumerate}
\end{prop}

\begin{proof}
Assertions (i) and (iii) are \cite[Lemma 7.3]{ds}, and (ii) is \cite[Lemma 7.2]{ds}.
\end{proof}

Let $\Delta_E$ denote the discriminant of some Weierstrass model of $E$.

\begin{prop}
\label{3.5}
If $\ell = 2$, $v \nmid 2\infty$, $E$ has multiplicative reduction at $v$, 
$\ord_v(\Delta_E)$ is odd, and 
$L_v/K_v$ is unramified, then $\H_\ell(L_v/K_v) = \H_\ell(K_v/K_v)$.
\end{prop}

\begin{proof}
This is \cite[Lemma 2.10(iii)]{H10}.
\end{proof}

\begin{defn}
\label{seldef}
The {\em relative Selmer group} $\Sel(L/K,E[\ell])$ is the subset of $H^1(K,E[\ell])$ defined by
$$
\Sel(L/K,E[\ell]) := \{c \in H^1(K,E[\ell]) : \text{$\loc_v(c) \in \H_\ell(L_v/K_v)$ for every $v$}\}
$$
where $\loc_v : H^1(K,E[\ell]) \to H^1(K_v,E[\ell])$ is the localization map at $v$.  
\end{defn}

When $L = K$, Proposition \ref{hl}(ii) shows that 
$\Sel(L/K,E[\ell])$ is the standard $\ell$-Selmer group of $E/K$, and we denote it by 
$\Sel(K,E[\ell])$.

When $\ell = [L:K] = 2$, $\Sel(L/K,E[\ell])$ is the standard $2$-Selmer group of the 
quadratic twist $E^L/K$ (see \cite[Lemma 8.4]{ds}).

These relative Selmer groups are useful to us because of the following proposition.

\begin{prop}
\label{ysel}
\textup{(i)} If $\Sel(L/K,E[\ell]) = 0$ then $\rk\,E(L) = \rk\,E(K)$.
\begin{enumerate}
\setcounter{enumi}{1}
\item
If $\rk\,E(L) = \rk\,E(K)$ and $L/K$ is ramified at two primes of good reduction 
for $E$ with different residue characteristics, then $E(L) = E(K)$.   
\end{enumerate}
\end{prop}

\begin{proof}
By \cite[Proposition 8.8]{ds}, 
$$
\rk\,E(K) \le \rk\,E(L) \le \rk\,E(K) + (\ell-1)\dim_{\Fl}\Sel(L/K,E[\ell]).
$$
Assertion (i) follows directly.

Suppose now that $\rk\,E(L) = \rk\,E(K)$ but $E(L)$ properly contains $E(K)$.  
Then we can fix a point $x \in E(L)$ such that $x \notin E(K)$ but $px \in E(K)$ for 
some rational prime $p$.  Since $x \notin E(K)$ and $[L:K]$ is prime, we must have $K(x) = L$.
But $K(x)/K$ can ramify only at places of bad reduction and primes above $p$
\cite[Theorem VII.7.1]{silv}, so this 
contradicts our assumption that $L/K$ ramifies at good primes with two distinct 
residue characteristics.
\end{proof}

\begin{lem}
\label{cocyclem}
Suppose that $c$ is a cocycle representing a nonzero class in $H^1(K,E[\ell])$.  
Let $F = K(E[\ell])$.  
The restriction of $c$ to $G_F$ induces a surjective homomorphism
$$
G_F \too E[\ell].
$$
\end{lem}

\begin{proof}
Using \eqref{ess}, the kernel of restriction $H^1(K,E[\ell]) \to H^1(F,E[\ell])$ 
is $H^1(F/K,E[\ell]) \cong H^1(\GL_2(\Fl),\Fl^2) = 0$.
Therefore the restriction of $c$ to $G_F$ is a 
nonzero homomorphism $\tilde{c} : G_F \to E[\ell]$.  
Since $\tilde{c}$ is the restriction of a class defined over $K$, 
we have that $\tilde{c}$ is $G_K$-equivariant, and in particular 
the image of $\tilde{c}$ is stable under $G_K$.  By \eqref{ess} it follows that 
$\tilde{c}$ is surjective.  
\end{proof}

\begin{defn}
If $\a$ is an ideal of $\O_K$,  
define relaxed-at-$\a$ and strict-at-$\a$ Selmer groups, respectively, by 
\begin{align*}
\Sel(K,E[\ell])^\a &:= \{c \in H^1(K,E[\ell]) : 
   \text{$\loc_v(c) \in \H_{\ell}(K_v/K_v)$ for every $v \nmid \a$}\},\\
\Sel(K,E[\ell])_\a &:= \{c \in \Sel(K,E[\ell]) : \text{$\loc_v(c) = 0$ for every $v \mid \a$}\}.
\end{align*}
Note that
$$
\Sel(K,E[\ell])_\a \subset \Sel(K,E[\ell]) \subset \Sel(K,E[\ell])^\a.
$$
\end{defn}

If $\Sigma$ is a finite set of places of $K$ containing all archimedean places, 
then the ring of $\Sigma$-integers of $K$ is
$$
\O_{K,\Sigma} := \{x \in K : \text{$x \in \O_{K_v}$ for every $v \notin \Sigma$}\}.
$$

\begin{defn}
\label{Qdef}
From now on let $\Sigma$ be a finite set of places of $K$ containing
all places where $E$ has bad reduction, all places dividing $\ell\infty$,
and large enough so that 
\begin{itemize}
\item
the primes in $\Sigma$ generate the ideal class group of $K$,
\item
the natural map $\O_{K,\Sigma}^\times/(\O_{K,\Sigma}^\times)^\ell \to 
   \prod_{v\in\Sigma} K_v^\times/(K_v^\times)^\ell$ is injective
\end{itemize}
(this is possible by \cite[Lemma 6.1]{kmr1}).
Define a set $\P$ of primes of $K$ by
$$
\P := \{\p \notin \Sigma : \bN\p \equiv 1 \pmod{\ell}\}
$$
and define a partition of $\P$ into disjoint subsets $\P_i$ for $0 \le i \le 2$ by
$$
\P_i := \{\p\in\P : \dim_{\Fl}\Hu(K_\p,E[\ell]) = i\}.
$$
(Equivalently by Proposition \ref{hl2}, $\P_i := \{\p\in\P : \dim_{\Fl}E(K_\p)[\ell]) = i\}$.) 
If $\a$ is an ideal of $\OK$, let $\P_1(\a)$ be the set of all $\p \in \P_1$ such that 
the localization map
$$
\Sel(K,E[\ell])_{\a} \map{\loc_\p} \Hu(K_\p,E[\ell])
$$
is nonzero.
\end{defn}

The next proposition is a modification of \cite[Proposition 9.10]{ds}.

\begin{prop}
\label{goodp}
\begin{enumerate}
\item
The sets $\P_0$ and $\P_1$ have positive density.
\item
Suppose $\a$ is an ideal of $\OK$ such that $\Sel(K,E[\ell])_\a$ is nonzero.  
Then $\P_1(\a)$ has positive density, and if $\p\in\P_1(\a)$ then 
$$
\dim_{\Fl}\Sel(K,E[\ell])_{\a\p} = \dim_{\Fl}\Sel(K,E[\ell])_{\a} - 1.
$$
\end{enumerate}
\end{prop}

\begin{proof}
Let $F := K(E[\ell])$, and 
fix $i = 0$ or $1$.
It follows from the surjection \eqref{ess} that $\Gal(F/K(\bmu_\ell)) \cong \SL_2(\Fl)$. 
Fix $\tau_i \in G_{K(\bmu_\ell)}$ such that $\dim_{\Fl}E[\ell]^{\tau_i = 1} = i$.

Suppose that $\p$ is a prime of $K$ whose Frobenius conjugacy class in 
$\Gal(K(E[\ell])/K)$ is the class of $\tau_i$.  
Since $\tau_i$ fixes $\bmu_{\ell}$, we have that $\bmu_{\ell} \subset K_\p^\times$
so $\bN\p \equiv 1 \pmod{\ell}$ and therefore by definition $\p\in\P$.  By Proposition \ref{hl2} 
we have 
$\dim_{\Fl} \H_\ell(L_v/K_v) = \dim_{\Fl}\Hu(K_v,E[\ell]) = \dim E[\ell]^{\tau_i = 1} = i$, 
so $\p\in\P_i$.  It follows from the Cebotarev Theorem that $\P_i$ has positive density.  
This is (i).

Fix an ideal $\a$ of $\OK$ and suppose that $c$ is a cocycle representing 
a nonzero element of $\Sel(K,E[\ell])_\a$.  Let $\tau_1$ be as above.
By Lemma \ref{cocyclem}, the restriction of $c$ to $G_F$ 
induces a surjective (and therefore nonzero) homomorphism 
$$
\tilde{c} : G_F \too E[\ell]/(\tau_1-1)E[\ell].
$$
Since $\tilde{c} \ne 0$, we can find $\gamma \in G_F$ such that 
$\tilde{c}(\gamma) \ne -c(\tau_1)$ in $E[\ell]/(\tau_1-1)E[\ell]$.  
Then 
$$
c(\gamma\tau_1) \notin (\tau_1-1)E[\ell] = (\gamma\tau_1-1)E[\ell].
$$

Let $N$ be a Galois extension of $K$ containing $F$ and such that the restriction 
of $c$ to $G_F$ factors through $\Gal(N/F)$.
If $\p$ is a prime whose Frobenius conjugacy class in $\Gal(N/K)$ is the class of 
$\gamma\tau_1$, then by Proposition \ref{hl2}(iii) we have $\loc_\p(c) \ne 0$, so 
$\p\in\P_1(\a)$.  Now the Cebotarev Theorem shows that $\P_1(\a)$ has positive density.

If $\p\in\P_1(\a)$ then we have an exact sequence
$$
0 \too \Sel(K,E[\ell])_{\a\p} \too \Sel(K,E[\ell])_{\a} \map{\loc_\p} \Hu(K_\p,E[{\ell}]) \too 0
$$
where the right-hand map is surjective because it is nonzero and the target space 
is one-dimensional.  This completes the proof of (ii).
\end{proof}

\begin{defn}
\label{updef}
Suppose $T$ is a finite set of primes of $K$, disjoint from $\Sigma$.  
If $\ell=2$ let $\Sigma_0$ denote the subset of $\Sigma$ consisting of all primes 
$\p$ where $E$ has multiplicative reduction and such that $\ord_\p(\Delta_E)$ is odd. 
If $\ell>2$ let $\Sigma_0$ be the empty set.
We say that an extension $L/K$ is {\em $T$-ramified and $\Sigma$-split} if 
\begin{itemize}
\item
every $\p\in T$ is ramified in $L/K$, every $\p \notin T$ is unramified in $L/K$, 
\item
every $v \in \Sigma - \Sigma_0$ splits in $L/K$.
\end{itemize}
\end{defn}

The next proposition is a modification of \cite[Proposition 9.17]{ds}.

\begin{prop}
\label{l7.13}
Let $r := \dim_{\Fl}\Sel(K,E[\ell])$ and suppose $t \le r$.
\begin{enumerate}
\item
There is a set of primes $\TT \subset \P_1$ of cardinality $t$ such that 
$$
\dim_{\Fl}\Sel(K,E[\ell])_\a = r-t,
$$
where $\a := {\prod_{\p\in T}\p}$.
\item
If $\TT$ satisfies (i), $T_0$ is a finite subset of $\P_0$, 
and $L/K$ is a cyclic extension of $K$ 
of degree $\ell$ that is $(T \cup T_0)$-ramified and $\Sigma$-split, then 
$$
\dim_{\Fl}\Sel(L/K,E[\ell]) = r-t.
$$
\end{enumerate}
\end{prop}

\begin{proof}
We will prove (i) by induction on $t$.  If $t = 0$, then $T$ is the empty set.

Suppose $\TT$ satisfies (i) for $t$, and $t < r$. 
Let $\a := {\prod_{\p\in \TT}\p}$.  
By Proposition \ref{goodp}(ii) we can find $\p \in \P_1(\a)$ 
so that 
$$
\dim_{\Fl}\Sel(K,E[\ell])_{\a\p} = r-t-1.
$$  
Then $\TT \cup \{\p\}$ satisfies (i) for $t+1$.  This proves (i).

Now suppose that $\TT$ satisfies (i), and let $\a := {\prod_{\p\in \TT}\p}$.  Consider the exact sequences
\begin{equation}
\label{gdd}
\raisebox{19pt}{
\xymatrix@C=12pt@R=7pt{
0 \ar[r] & \Sel(K,E[\ell]) \ar[r] & \Sel(K,E[\ell])^{\a} \ar^-{\oplus \loc_\p}[rr] 
   && \dirsum{\p\in \TT}\displaystyle\frac{H^1(K_\p,E[\ell])}{\Hu(K_\p,E[\ell])} \\
0 \ar[r] & \Sel(K,E[\ell])_{\a} \ar[r] & \Sel(K,E[\ell]) \ar^-{\oplus \loc_\p}[rr] 
   && \dirsum{\p\in \TT}\Hu(K_\p,E[\ell]).
}}
\end{equation}
By global duality (see for example \cite[Theorem 2.3.4]{kolysys}),
the images of the two right-hand maps in \eqref{gdd} are orthogonal complements of each other 
under the sum of the local Tate pairings.  By our choice of $\TT$ the lower right-hand map 
is surjective, so the upper right-hand map is zero, i.e., 
\begin{equation}
\label{siss}
(\oplus_{\p\in \TT} \loc_\p)(\Sel(K,E[\ell])^{\a})  \subset \dirsum{\p\in \TT}\Hu(K_\p,E[\ell]).
\end{equation}
Let $T_0$ be a finite subset of $\P_0$, let $\b := \prod_{\p\in T_0}\p$, and suppose 
$L$ is a cyclic extension of $K$ of degree $\ell$ that is 
$(T \cup T_0)$-ramified and $\Sigma$-split.
By Propositions \ref{hl2}(i) and \ref{3.5}, we have $\H_\ell(L_v/K_v) = \H_\ell(K_v/K_v)$
if $v \notin T \cup T_0$.  Thus 
by Definition \ref{seldef}, $\Sel(L/K,E[\ell])$ is the kernel of the map
$$
\Sel(K,E[\ell])^{\a\b} \map{\oplus_{\p\in T \cup T_0} \loc_\p} 
   \dirsum{\p\in T \cup T_0}H^1(K_\p,E[\ell])/\H_\ell(L_\p/K_\p).
$$
We have $\H_\ell(L_\p/K_\p) = H^1(K_\p,E[\ell]) = 0$ for every $\p\in\P_0$ 
by Propositions \ref{hl}(i) and \ref{hl2}(ii) and the 
definition of $\P_0$, so in fact $\Sel(L/K,E[\ell])$ is the kernel of the map
\begin{equation}
\label{siss2}
\Sel(K,E[\ell])^\a \map{\oplus_{\p\in \TT} \loc_\p} \dirsum{\p\in \TT}H^1(K_\p,E[\ell])/\H_\ell(L_\p/K_\p).
\end{equation}
By Proposition \ref{hl1}, $\Hu(K_\p,E[\ell]) \cap \H_\ell(L_\p/K_\p) = 0$ for every $\p \in T$.  
Combining \eqref{siss} and \eqref{siss2} shows that $\Sel(L/K,E[\ell]) = \Sel(L/K,E[\ell])_\a$, 
so by our choice of $\TT$ we have $\dim_{\Fl}\Sel(L/K,E[\ell]) = r-t$.
This proves (ii).
\end{proof}

\begin{prop}
\label{3.16}
Suppose $T$ is a finite subset of $\P_0 \cup \P_1$.  
If $\ell = 2$ suppose further that $E$ has a prime $\q\nmid 2$ of multiplicative reduction 
such that $\ord_\q(\Delta_E)$ is odd.  
Then there are a finite 
subset $T_0 \subset \P_0$ and a cyclic extension $L/K$ of degree $\ell$ such that 
\begin{enumerate}
\item
$L/K$ is $(T \cup T_0)$-ramified and $\Sigma$-split,
\item
$K$ has a prime $\p$ of degree $1$, unramified over $\Q$, such that $L/K$ is ramified at $\p$ and 
unramified at all primes $\p' \ne \p$ with the same residue characteristic as $\p$.
\end{enumerate} 
\end{prop}

\begin{proof}
Fix a prime $\p \in \P_0$ of degree $1$, unramified over $\Q$, 
whose residue characteristic $p$ is different from 
the residue characteristics of all primes in $T$ (this is possible by 
Proposition \ref{goodp}(i)).  Let $\Sigma_p := \Sigma \cup \{\text{$v$ of $K$ : $v \mid p$}\}$.

Define a set of global Galois characters
\begin{multline*}
\cC(T) := \{\chi \in \Hom(G_K,\bmu_\ell) : \text{$\chi$ is ramified at all $\p \in T$ and} \\ 
\text{unramified at all primes not in $\Sigma_p \cup T \cup \P_0$}\},
\end{multline*}
and a set of tuples of local characters 
$$
\Omega_T := \prod_{v \in \Sigma_p} \Hom(G_{K_v}, \bmu_\ell) 
   \times \prod_{v \in T} \Hom_\ram(G_{K_v}, \bmu_\ell)
$$
where $\Hom_\ram(G_{K_v}, \bmu_\ell) \subset \Hom(G_{K_v}, \bmu_\ell)$ denotes the 
subset of ramified characters.

Suppose first that $\ell \ge 3$.  
Restriction gives a natural map of sets $\cC(T) \to \Omega_T$, and by 
\cite[Proposition 10.7]{kmr2} this map is surjective.  
Now take an element $\omega \in \Omega_T$ whose $v$-component $\omega_v$ is 
the trivial character if $v \in \Sigma_p$ is different from $\p$, 
and such that $\omega_\p$ is ramified.
Let $\chi \in \cC(T)$ be any character that restricts to $\omega$, and let $T_0$ 
be the set of primes (necessarily in $\P_0$) not in $T$ where $\chi$ is ramified. 
If $L$ is the cyclic extension of $K$ corresponding to $\chi$, then $L/K$ 
is $(T \cup T_0)$-ramified and $\Sigma$-split, and $L$ ramifies at $\p$ 
but not at any other prime above $p$.  This proves the proposition when $\ell \ge 3$.

When $\ell = 2$ the proof is similar, except that the map $\eta : \cC(T) \to \Omega_T$ 
is not surjective.    
However in this case, \cite[Proposition 10.7]{kmr2} shows that the image of $\eta$ 
contains either our chosen $\omega \in \Omega_T$, or else 
$\omega'$, where $\omega'_v$ is $\omega_v$ if 
$v \ne \q$, and $\omega'_\q$ is the nontrivial unramified quadratic character 
(where $\q$ is the given prime of multiplicative reduction).  
The proof now proceeds exactly as in the case of odd $\ell$, using a character 
$\chi \in \cC(T)$ that maps to either $\omega$ or $\omega'$.
\end{proof}

\section{Proof of Theorem \ref{main}}

\begin{defn}
Let $\cS$ be the set of all elliptic curves $E$ over $\Q$ satisfying all 
of the following properties:
\begin{itemize}
\item
for every prime $\ell$, the map $G_{\Q} \to \Aut(E[\ell]) \cong \GL_2(\F_\ell)$ 
is surjective,
\item
$E$ has discriminant $\Delta_E \equiv 1 \pmod{4}$, 
\item
$E$ has a prime $q \nmid 2$ of multiplicative reduction such that 
$\ord_q(\Delta_E)$ is odd.
\end{itemize}
Let $\cS_0 = \{E \in \cS : E(\Q) = 0\}$.
\end{defn}

\begin{lem}
\label{Sne}
The set $\cS_0$ is nonempty.
\end{lem}

\begin{proof}
Let $E$ be the elliptic curve labelled $67.a1$ in the $L$-functions and Modular Forms Database 
\cite{lmfdb}.  
We will show that $E \in \cS_0$.

The surjectivity of the map $G_\Q \to \Aut(E[\ell]) \cong \GL_2(\F_\ell)$
for every $\ell$ is stated in \cite[curve $67.a1$]{lmfdb}; alternatively this follows directly from 
\cite[Proposition 21]{serre} and the fact that $E$ has no rational isogenies \cite[curve $67.a1$]{lmfdb}.
The curve $E$ has multiplicative reduction at $67$, and its discriminant $-67$ is 
congruent to $1 \pmod{4}$.  Thus $E \in \cS$, and $E(\Q) = 0$ so $E \in \cS_0$.
\end{proof}

\begin{rem}
\label{rmk1}
The elliptic curve $67.a1$ in the proof of Lemma \ref{Sne} has 
only one rational point.  If we instead take $E$
to be the curve $37.a1$, then $E(\Q) \cong \Z$ and a similar argument shows that $E \in \cS$.
\end{rem}

For $E \in \cS$, define $T_E$ to be the compositum 
of all the fields $\Q(E[\ell])$ for all primes $\ell$.  Note that if $K$ is a 
number field and $K \cap T_E = \Q$, then 
for every prime $\ell$, the map $G_K \to \Aut(E[\ell]) \cong \GL_2(\F_\ell)$ 
is surjective.

\begin{prop}
\label{prop}
Suppose $E \in \cS$, and
$K$ is a finite nontrivial extension of $\Q$ such that $K \cap T_E = \Q$ and $E(K) = E(\Q)$.  Then 
for each rational prime $\ell$ there are infinitely many cyclic extensions $L/K$ 
of degree $\ell$ such that:
\begin{enumerate}
\item
$E(L) = E(K)$,
\item
$K$ has a prime $\p$ of degree $1$, unramified over $\Q$, such that $L/K$ is ramified at $\p$
and unramified at all primes $\p' \ne \p$  with the same residue characteristic as $\p$,
\item
$L/K$ is ramified at (at least) two primes where $E$ has good reduction and that have 
different residue characteristics,
\item
every place of $K$ dividing $\ell\infty$ splits completely in $L/K$.
\end{enumerate}
\end{prop}

\begin{proof}
Let $T$ be a finite subset of $\P_1$ satisfying Proposition \ref{l7.13}(i) 
with $t := \dim_{\Fl}\Sel(K,E[\ell])$.  Let $T'$ be a finite subset of $T \cup \P_0$ 
containing $T$ and at least two primes in $\P_0$ with different residue characteristics.
Now apply Proposition \ref{3.16}, with the set $T'$ in place of $T$, to produce 
a cyclic extension $L/K$ that is $(T' \cup T_0)$-ramified and $\Sigma$-split for 
some $T_0 \subset \P_0$.

By Proposition \ref{l7.13}(ii) we have $\Sel(L/K,E[\ell]) = 0$, and then 
Proposition \ref{ysel} shows that $E(L) = E(K)$.  
Assertion (ii) is Proposition \ref{3.16}(ii), and (iii) and (iv) 
follow from our choice of $T'$ and Definition \ref{updef} 
of ``$\Sigma$-split'', since all places dividing $\ell\infty$ are in $\Sigma$.

Thus $L$ has the desired properties.  By varying the set $T'$ we can produce 
infinitely many such $L$.  
\end{proof}

\begin{lem}
\label{llem}
Suppose $E \in \cS$.  Then there is a real quadratic field $F$, ramified 
at $2$, such that $E(F) = E(\Q)$ and $F \cap T_E  = \Q$.
\end{lem}

\begin{proof}
If $A$ is an elliptic curve over $\Q$ and $d \in \Q^\times$, let $A^{(d)}$ denote 
the quadratic twist of $A$ corresponding to $\Q(\sqrt{d})/\Q$.  Then 
$(A^{(d)})^{(d')} = A^{(dd')}$ and 
\begin{equation}
\label{twistsum}
\rk\, A(\Q(\sqrt{d})) = \rk\, A(\Q) + \rk\, A^{(d)}(\Q).
\end{equation}

Fix an elliptic curve $E \in \cS$.  Apply Proposition \ref{prop} to the curve 
$E^{(2)}$ with $K = \Q$ and $\ell = 2$ to get a quadratic extension $L/\Q$ 
satisfying (i) through (iv) of Proposition \ref{prop}.  
Write $L = \Q(\sqrt{D})$ with a squarefree integer $D$, and put $F := \Q(\sqrt{2D})$. 

Applying \eqref{twistsum} with $A = E^{(2)}$ and $d = D$ yields
$$
\rk\, E^{(2)}(L) = \rk\, E^{(2)}(\Q) + \rk\, E^{(2D)}(\Q).
$$
By Proposition \ref{prop}(i) we have $E^{(2)}(L) = E^{(2)}(\Q)$, so $\rk\, E^{(2D)}(\Q) = 0$.
Applying \eqref{twistsum} with $A = E$ and $d = 2D$ yields
$$
\rk\, E(F) = \rk\, E(\Q) + \rk\, E^{(2D)}(\Q)
$$
so $\rk\, E(F) = \rk\, E(\Q)$.  By Proposition \ref{prop}(iii,iv) $D$ is positive, odd, 
and divisible by at least $2$ primes where $E$ has good reduction.  Thus $F$ is a real field, 
ramified at $2$, and $E(F) = E(\Q)$ by Proposition \ref{ysel}(ii).

Since the discriminant $\Delta_E \equiv 1 \pmod{4}$, the quadratic 
field $\Q(\sqrt{\Delta_E}) \subset \Q(E[2])$ is unramified at $2$, so  
$2$ is tamely ramified in $\Q(E[2])/\Q$.  Since $\Delta_E$ is odd, $E$ has good 
reduction at $2$, so $2$ is unramified in $\Q(E[\ell])/\Q$ for all 
odd primes $\ell$ \cite[Theorem VII.7.1]{silv}.  
Hence $2$ is tamely ramified in $T_E/\Q$, but $2$ is wildly ramified in $F/\Q$,
so $F \cap T_E = \Q$.
\end{proof}

\begin{thm}
\label{thm}
Suppose $E \in \cS$.
Then for every infinite sequence $\L$ of rational primes with $\ell_1 = 2$, 
there are uncountably many totally real 
$\L$-towers $K_\infty/\Q$ such that $E(K_\infty) = E(\Q)$.
\end{thm}

\begin{proof}
We build an $\L$-tower inductively as follows.  Let $K_0 = \Q$ and let $K_1$ be 
a field $F$ as in Lemma ]\ref{llem}.
Now for each $i \ge 2$ apply Proposition \ref{prop} to produce a cyclic extension 
$K_i/K_{i-1}$ of degree $\ell_i$ such that $E(K_i) = E(K_{i-1}) = E(\Q)$.

Since $K_1 \cap T_E = F\cap T_E = \Q$, Lemma \ref{towerlem}(ii) shows that $K_i \cap T_E = \Q$ and we can 
continue by induction.  The result is an $\L$-tower $K_\infty/\Q$ such that 
$E(K_\infty) = E(\Q)$.  
Since at each step Proposition \ref{prop} provides us with infinitely many possible $K_i$, 
and different choices give rise to distinct $\L$-towers by Corollary \ref{cor},
we obtain uncountably many $\L$-towers in this way.
\end{proof}

\begin{proof}[Proof of Theorem \ref{main}]
Let $\L$ be a sequence of rational primes such that 
\begin{itemize}
\item
$\ell_1 = 2$, and 
\item
every prime occurs infinitely many times in $\L$.  
\end{itemize}
If $K_\infty/\Q$ is an $\L$-tower, then $K_\infty$ is a big field.
We can apply Theorem \ref{thm}, taking for $E \in \cS_0$ the curve $67.a1$ from the proof 
of Lemma \ref{Sne}, to produce uncountably many $\L$-towers 
$K_\infty/\Q$ with $|E(K_\infty)| = |E(\Q)| = 1$; such a $K_\infty$ is a big field 
that is not large.  This concludes the proof of Theorem \ref{main}.
\end{proof}

\section{Additional remarks}\label{add}

\begin{rem}
Let $\K$ be the compositum of all $\Z_\ell$-extensions of $\Q$.  Then $\K$ is a big field, 
and we conjecture that $\K$ is not large.  

More precisely, let $E$ be an elliptic curve over $\Q$.  
In \cite[Conjecture 10.2]{expmath} we conjecture that $E(\K)$ is finitely generated.
Thus combining \cite[Conjecture 10.2]{expmath} with the following theorem 
of Fehm and Petersen leads to the conjecture that $\K$ is not large.

\begin{thm}[Fehm and Petersen, Theorem 1.2 of \cite{FP}] 
\label{11}
If $\K \subset {\bar{\Q}}$ is a large field and $A$ is an abelian variety over $\K$, 
then $A(\K)$ has infinite rank.
\end{thm}

Theorem \ref{11} was first proved by Tamagawa in the case that $A$ is an elliptic curve  
(see the remark at the bottom of page 580 of \cite{FP}).
\end{rem}

\begin{rem}
By  Main Theorem A of \cite{htp}, Hilbert's Tenth Problem 
has a negative answer for the ring of integers in any subfield $K$ of ${\bar{\Q}}$ 
satisfying
\begin{itemize} 
\item 
$K$ is totally real, and 
\item 
there is an elliptic curve $E$ over $K$ such that $E(K)$ 
is finitely generated and has positive rank.
\end{itemize} 
By Theorem \ref{11}, such a $K$ is not a large field.

Applying Theorem \ref{thm} with the elliptic curve $E = 37.a1$ as in Remark \ref{rmk1},
one can find uncountably many $\L$-towers $K_\infty/\Q$ with $K_\infty$ totally real 
and big such that $E(K_\infty) = E(\Q) \cong \Z$. 
This gives uncountably many non-large big fields over whose ring of integers 
Hilbert's Tenth Problem has a negative answer.

It is natural to ask whether there is {\em any} non-large field over whose 
ring of integers Hilbert's Tenth Problem has a positive answer.
\end{rem}

\end{document}